\documentclass[12pt]{article}
\usepackage{amssymb,amsmath,amsthm,amsfonts,mathrsfs}
\usepackage{enumerate}

\theoremstyle{definition}

\usepackage{pgf,tikz}
\usetikzlibrary{intersections}
\usepackage{graphicx}
\usepackage{float}
\usepackage[colorlinks=true, allcolors=blue]{hyperref}
\usepackage[all]{hypcap}
\usepackage[capitalize,nameinlink]{cleveref}
\usepackage{titlesec}
\titlelabel{\thetitle.\quad}
\newtheorem{theo}{Theorem}

\theoremstyle{definition}
\newtheorem{case}{Case}[theo]

\newtheorem{definition}{Definition}
\begin{document}
\title{New Vertex Ordering Characterizations of Circular-Arc Bigraphs}
\def\correspondingauthor{\footnote{Corresponding author}}
\author{ Indrajit Paul, Ashok Kumar Das\correspondingauthor{}\\Department of Pure Mathematics, University of Calcutta\\
Email Address -  
paulindrajit199822@gmail.com \&\\ ashokdas.cu@gmail.com}
\maketitle
\begin{abstract}
  In this article, we present two new characterizations of circular-arc bigraphs based on their vertex ordering. Also, we provide a characterization of circular-arc bigraphs in terms of forbidden patterns with respect to a particular ordering of their vertices.                                                                              
\end{abstract}
\noindent {\bf Keywords:}
circular-arc bigraphs, vertex ordering, total-circular ordering, bi-circular ordering , forbidden pattern
\section{Introduction}
A graph $G=(V,E)$ is a circular-arc graph if it is the intersection graph of circular-arcs of a host circle. A bipartite graph (in short, bigraph) $B=(X,Y,E)$ is a circular-arc bigraph if there exists a family $\mathcal{A}=\{ A_v : v\in X\cup Y\}$ of circular-arcs such that $uv\in E$ if and only if $A_u\cap A_v\neq \phi$, where $u\in X$, $v\in Y$.\\
The problem of characterization of circular-arc graphs was initiated by Klee \cite{klee}. Circular-arc graphs and their subclasses like proper circular-arc graphs (if there exist no two arcs in the representation that one is properly contained in another), Helly circular-arc graphs (if the circular-arc representation satisfy the Helly property) have been extensively studied by Tucker and others \cite{Gavril, tucker1, tucker, tucker2}. Very recently an obstruction characterization and certifying recognition algorithm for circular-arc graphs have been given by M.Francis, P.Hell and J.Stacho \cite{fhs}. But the bipartite version of the circular-arc graphs, i.e. circular-arc bigraphs remains a relatively less explored field. Sen et al. \cite{sdw} introduced circular-arc di/bigraphs.  They \cite{bdgs,sdw} also gave several characterizations of circular-arc bigraphs. Das, Chakraborty \cite{dc} and Safe \cite{safe} studied proper circular-arc bigraphs. Most of the characterizations are based on the adjacency matrix.  In this paper first we provide two characterizations of circular-arc bigraphs based on their vertex ordering. Hell and Huang \cite{hell} characterized interval bigraphs ( intersection bigraphs of  family of intervals) using forbidden patterns with respect to particular vertex ordering. Motivated by this result, we give an analogous characterization of circular-arc bigraphs in terms of forbidden patterns with respect to a specific vertex ordering.

\section{Main Result}
In this section, we introduce various types of vertex orderings for bigraphs and explore their role in characterizing circular-arc bigraphs. We provide  characterizations of circular-arc bigraphs based on these orderings. Furthermore, we establish a characterization using forbidden patterns, which could be one of the most fascinating characterizations of circular-arc bigraphs to date.\\
First, we will define the \textit{total-circular ordering} of  the vertices of a bigraph.
\begin{definition}\label{d1}
Consider a bipartite graph $B=(X,Y,E)$ of order $n$. Order the vertices of $B$ from $1$ to $n$ and arrange them on an $n$-hour clock, such that the $i^{\text{th}}$ vertex is on the $i^{\text{th}}$ hour marker. Assume that the vertex set $X\cup Y$ satisfyies the following conditions:
\begin{enumerate}[(a)]

    \item $x_iy_j\in E$ ($i>j$) implies
    \begin{itemize}
        \item either $x_iy_k\in E$, for all possible $y_k$, where $k\in \{ j+1, j+2,...,i-2,i-1\}$
        \item or, $x_ly_j\in E$, for all possible $x_l$, where $l\in \{i+1,i+2,...,n,1,2,...,j-1\}$,
    \end{itemize}
     \item $x_iy_j\in E$ ($i<j$) implies
    \begin{itemize}
        \item either $x_ky_j\in E$, for all possible $x_k$, where $k\in \{ i+1, i+2,...,j-2,j-1\}$
        \item or, $x_iy_l\in E$, for all possible $y_l$, where $l\in \{j+1,j+2,...,n,1,2,...,i-1\}$.
    \end{itemize}
\end{enumerate}
Then the vertex set $X\cup Y$ of $B$ is said to have a \textit{ total-circular ordering}.\\
 Using the total-circular ordering of a bipartite graph, we will characterize circular-arc bigraphs in the following theorem:
\end{definition}
\begin{theo}\label{t3}
    \textit{A bigraph $B=(X,Y,E)$ is a circular-arc bigraph if and only if the vertex set $X\cup Y$ of $B$ has a \textit{total-circular ordering}}.
   
\end{theo}
\begin{proof}
    Necessity: Let $B=(X,Y,E)$ be a circular-arc bigraph. Then there exist a circular arc $A_v$ corresponding to every vertex $v$ of the set $X\cup Y$. Such that $xy\in E$ if and only if $A_x\cap A_y\neq\phi$, for all $x\in X$ and $y\in Y$. Without loss of generality we consider that all the arcs having distinct end points. Now order the vertices of $B$ from $1$ to $n$ according to increasing order of clockwise end points (where $n$ is the order of the bigraph $B$). Let $v_1, v_2, v_3, ..., v_n$ be such an ordering. If $i^{\text{th}}$ vertex of the sequence belongs to $X$ partite set then we call the vertex as $x_i$, i.e $v_i=x_i$. Similarly, if $j^{\text{th}}$ vertex of the sequence belongs to $Y$ partite set then we call the vertex  $y_j$, i.e $v_j=y_j$. 
    \par Let $x_i$ be adjacent to $y_j$. Then, we have the following cases:
    \begin{case}\label{c1}
        ($i>j$)
    \end{case}
    \begin{itemize}

   \item Either, clockwise end point of $A_i$ (arc corresponding to $x_i$) lies within $A_j$ ( arc corresponding to $y_j$), in this case $A_k\cap A_j\neq \phi$, for all $k\in \{i+1, i+2,...,n,1,...,j-1\}$.\\
   Therefore $x_ky_j\in E$ for all possible $k\in \{i+1, i+2,...,n,1,...,j-1\}$.
   \item Or, clockwise end point of $A_j$ lies within $A_i$, in this case $A_l\cap A_i\neq\phi$, for all $l\in\{ j+1, j+2,...,i-1\}$.\\
   Therefore $x_iy_l\in E$ for all possible $l\in\{ j+1, j+2,...,i-1\}$.
    \end{itemize}
    \begin{case}\label{c2}
        ($i<j$)
    \end{case}
    \begin{itemize}
    \item Either, clockwise end point of $A_i$ (arc corresponding to $x_i$) lies within $A_j$ ( arc corresponding to $y_j$), in this case $A_k\cap A_j\neq \phi$, for all $k\in \{i+1, i+2,...,j-1\}$.\\
   Therefore $x_ky_j\in E$ for all possible $k\in \{i+1, i+2,...,j-1\}$.
   \item Or, clockwise end point of $A_j$ lies within $A_i$, in this case $A_l\cap A_i\neq\phi$, for all $l\in\{ j+1, j+2,...,n,1,...,i-1\}$.\\
   Therefore $x_iy_l\in E$ for all possible $l\in\{ j+1, j+2,...,n,1,...,i-1\}$.
    \end{itemize}
    Hence, the ordering $v_1,v_2,...,v_n$ of  vertices of the bigraph $B$ is a total-circular ordering.
    \par\noindent{}Sufficiency: Let $B=(X,Y,E)$ be a bipartite graph where the vertices are ordered as $v_1,v_2,...,v_n$ , which is a total-circular ordering.\\
    Now, we will construct a circular arc for each vertex of  the bigraph $B$. Let $k$ be the $k^{\text{th}}$ hour marker on an n-hour clock.\\
    If $v_i=x_i$ $\in X$, then draw a closed arc $A_i$ anticlockwise from $i$ to $r_i$, where $v_{r_i}$ is the last consecutive vertex of $Y$-partite set in the anticlockwise sequence $v_{i-1},v_{i-2},...,v_i$ that is adjacent to $x_i$ (i.e. $A_i=[r_i, i]$).\\
     If $v_j=y_j$ $\in Y$, then draw a closed arc $A_j$ anticlockwise from $j$ to $r_j$, where $v_{r_j}$ is the last consecutive vertex of $X$-partite set in the anticlockwise sequence $v_{j-1},v_{j-2},...,v_j$ that is adjacent to $y_j$ (i.e. $A_j=[r_j, j]$).
     \par If $x_i$ is adjacent to $y_j$, for some $x_i\in X$ and $y_j\in Y$, then we have the following cases:
     \begin{case}\label{c3}
         ($i>j$)
     \end{case}
     Then, by \cref{d1} of total-circular ordering:
     \begin{itemize}
          \item either $x_iy_k\in E$, for all possible $y_k$ ($k\in \{ j+1, j+2,...,i-2,i-1\}$), then $A_i$ contains $j$ and therefore $A_i\cap A_j\neq\phi$.
          
        \item or, $x_ly_j\in E$, for all possible $x_l$ ($l\in \{i+1,i+2,...,n,1,2,...,j-1\}$), then $A_j$ contains $i$ and therefore $A_i\cap A_j\neq\phi$.
     \end{itemize}
     \begin{case}\label{c4}
         ($i<j$)
     \end{case}
     \begin{itemize}
        \item either $x_ky_j\in E$, for all possible $x_k$ ($k\in \{ i+1, i+2,...,j-2,j-1\}$), then $A_j$ contains $i$ and therefore $A_i\cap A_j\neq\phi$.
        \item or, $x_iy_l\in E$, for all possible $y_l$ ($l\in \{j+1,j+2,...,n,1,2,...,i-1\}$), then $A_i$ contains $j$ and therefore $A_i\cap A_j\neq\phi$.
    \end{itemize}
    Therefore in any case, $x_i$ is adjacent to $y_j$ implies $A_i\cap A_j\neq\phi$.\\
    Again, let $A_i\cap A_j\neq\phi$, where $A_i$ is the circular arc corresponding to the vertex $x_i$ and $A_j$ is the circular arc corresponding to the vertex $y_j$. Then by the construction of the circular arcs it is clear that the vertex $x_i$ is adjacent to $y_j$.\\
    Thus $x_iy_j\in E$ if and only if $A_i\cap A_j\neq\phi$.
    Therefore $B=(X,Y,E)$ is a circular-arc bigraph. 
    
\end{proof}

\begin{figure}[H]{\label{f1}}
    \centering
    \begin{tikzpicture}[line cap=round,line join=round,x=1.0cm,y=1.0cm,scale=.7]
\clip(-6.,1.) rectangle (11.,9.);
\fill[line width=.7pt,color=black,fill=white,fill opacity=0.10000000149011612] (-3.4,2.56) -- (-1.4,2.56) -- (0.014213562373094568,3.9742135623730945) -- (0.014213562373095012,5.974213562373094) -- (-1.4,7.388427124746189) -- (-3.4,7.3884271247461895) -- (-4.814213562373094,5.9742135623730945) -- (-4.814213562373094,3.9742135623730954) -- cycle;
\draw [line width=.7pt] (7.,5.) circle (2.23606797749979cm);
\draw [shift={(7.,5.)},line width=.7pt]  plot[domain=0.:2.562953072305964,variable=\t]({1.*3.*cos(\t r)+0.*3.*sin(\t r)},{0.*3.*cos(\t r)+1.*3.*sin(\t r)});
\draw [shift={(7.,5.)},line width=.7pt]  plot[domain=0.9514840476604486:2.1169628902286846,variable=\t]({1.*2.480725700274014*cos(\t r)+0.*2.480725700274014*sin(\t r)},{0.*2.480725700274014*cos(\t r)+1.*2.480725700274014*sin(\t r)});
\draw [shift={(7.,5.)},line width=.7pt]  plot[domain=-0.872357789965374:0.5201245875154695,variable=\t]({1.*2.768176294963888*cos(\t r)+0.*2.768176294963888*sin(\t r)},{0.*2.768176294963888*cos(\t r)+1.*2.768176294963888*sin(\t r)});
\draw [shift={(7.,5.)},line width=.7pt]  plot[domain=0.29145679447786715:1.4016951007935197,variable=\t]({1.*3.3408980828513757*cos(\t r)+0.*3.3408980828513757*sin(\t r)},{0.*3.3408980828513757*cos(\t r)+1.*3.3408980828513757*sin(\t r)});
\draw [shift={(7.,5.)},line width=.7pt]  plot[domain=4.536154282648014:5.73369576085061,variable=\t]({1.*2.96593998590666*cos(\t r)+0.*2.96593998590666*sin(\t r)},{0.*2.96593998590666*cos(\t r)+1.*2.96593998590666*sin(\t r)});
\draw [shift={(7.,5.)},line width=.7pt]  plot[domain=3.8417705860151123:5.0318822022723415,variable=\t]({1.*2.824535360019414*cos(\t r)+0.*2.824535360019414*sin(\t r)},{0.*2.824535360019414*cos(\t r)+1.*2.824535360019414*sin(\t r)});
\draw [shift={(7.,5.)},line width=.7pt]  plot[domain=3.0435473762959537:4.1623146077095,variable=\t]({1.*2.451774867315513*cos(\t r)+0.*2.451774867315513*sin(\t r)},{0.*2.451774867315513*cos(\t r)+1.*2.451774867315513*sin(\t r)});
\draw [shift={(7.,5.)},line width=.7pt]  plot[domain=2.3220197665761497:3.4141405272160372,variable=\t]({1.*3.3111931384321265*cos(\t r)+0.*3.3111931384321265*sin(\t r)},{0.*3.3111931384321265*cos(\t r)+1.*3.3111931384321265*sin(\t r)});
\draw [line width=.7pt,color=black] (-3.4,2.56)-- (-1.4,2.56);
\draw [line width=.7pt,color=black] (-1.4,2.56)-- (0.014213562373094568,3.9742135623730945);
\draw [line width=.7pt,color=black] (0.014213562373094568,3.9742135623730945)-- (0.014213562373095012,5.974213562373094);
\draw [line width=.7pt,color=black] (0.014213562373095012,5.974213562373094)-- (-1.4,7.388427124746189);
\draw [line width=.7pt,color=black] (-1.4,7.388427124746189)-- (-3.4,7.3884271247461895);
\draw [line width=.7pt,color=black] (-3.4,7.3884271247461895)-- (-4.814213562373094,5.9742135623730945);
\draw [line width=.7pt,color=black] (-4.814213562373094,5.9742135623730945)-- (-4.814213562373094,3.9742135623730954);
\draw [line width=.7pt,color=black] (-4.814213562373094,3.9742135623730954)-- (-3.4,2.56);
\draw [line width=.7pt] (-3.4,7.3884271247461895)-- (0.014213562373094568,3.9742135623730945);
\draw (-3.7,8.2) node[anchor=north west] {$y_3$};
\draw (-1.7,8.2) node[anchor=north west] {$x_1$};
\draw (0.2,6.32) node[anchor=north west] {$y_2$};
\draw (0.2,4.36) node[anchor=north west] {$x_4$};
\draw (-1.7,2.58) node[anchor=north west] {$y_5$};
\draw (-3.7,2.62) node[anchor=north west] {$x_6$};
\draw (-5.9,4.38) node[anchor=north west] {$y_7$};
\draw (-5.9,6.44) node[anchor=north west] {$x_8$};
\draw (7.8,7.8) node[anchor=north west] {$x_1$};
\draw (10.2,6.5) node[anchor=north west] {$y_2$};
\draw (10.0,5.5) node[anchor=north west] {$y_3$};
\draw (8.2,3.6) node[anchor=north west] {$x_4$};
\draw (6.2,2.1) node[anchor=north west] {$y_5$};
\draw (4.1,3.3) node[anchor=north west] {$x_6$};
\draw (3.7,5.64) node[anchor=north west] {$y_7$};
\draw (3.7,7.9) node[anchor=north west] {$x_8$};
\begin{scriptsize}
\draw [fill=black] (-3.4,2.56) circle (2.5pt);
\draw [fill=black] (-1.4,2.56) circle (2.5pt);
\draw [fill=black] (0.014213562373094568,3.9742135623730945) circle (2.5pt);
\draw [fill=black] (0.014213562373095012,5.974213562373094) circle (2.5pt);
\draw [fill=black] (-1.4,7.388427124746189) circle (2.5pt);
\draw [fill=black] (-3.4,7.3884271247461895) circle (2.5pt);
\draw [fill=black] (-4.814213562373094,5.9742135623730945) circle (2.5pt);
\draw [fill=black] (-4.814213562373094,3.9742135623730954) circle (2.5pt);
\end{scriptsize}
\end{tikzpicture}
    \caption{A circular-arc bigraph where a total-circular ordering of the vertices is: $x_1$, $y_2$, $y_3$, $x_4$, $y_5$, $x_6$, $y_7$, $x_8$.}
    \label{fig:enter-label-1}
\end{figure}
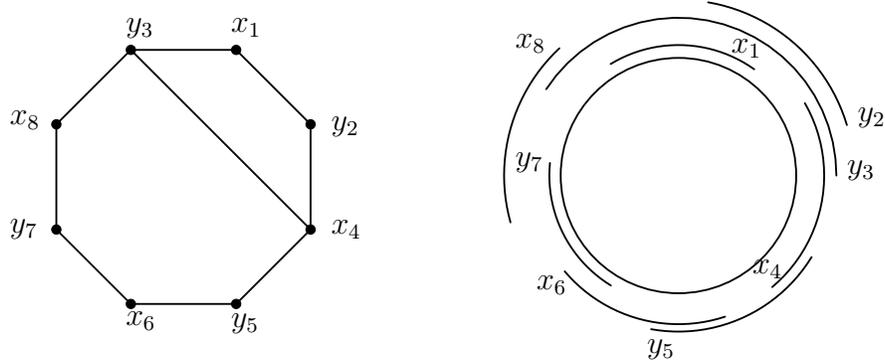
If we calculate the circular arcs of the given graph corresponding to the total-circular ordering of the vertices as shown in Figure 1, we get the following arcs: $A_{x_1}=[1,1]$, $A_{y_2}=[1,2]$, $A_{y_3}=[8,3]$, $A_{x_4}=[2,4]$, $A_{y_5}=[4,5]$, $A_{x_6}=[5,6]$, $A_{y_7}=[6,7]$, and  $A_{x_8}=[7,8]$.\vspace{.3cm}
\par We now introduce another vertex ordering for bigraphs, referred to as bi-circular ordering. Utilizing this ordering, we characterize the class of circular arc bigraphs.\\
Let $B=(X,Y,E)$ be a bigraph of order $n$. Let $v_1$, $v_2$,..., $v_n$ be an ordering of the vertex set $X\cup Y$ of $B$. These vertices are then placed on an $n$-hour clock such that the $i$-th vertex is placed on the $i$-th hour marker. Let $A$ denote the biadjacency matrix of the bigraph $B$, where the vertices in the rows and columns are arranged according to the increasing order of indices of vertices in $X$ and $Y$ partite sets respectively (see Figure 3).\\
Now, consider the row of $A$, corresponding to the vertex $x_i$.
We define the set $W_i$ as the collection of $1$’s in the $x_i$-th row that appear consecutively, starting from the column $y_{m_i}$, where $y_{m_i}$ is the first vertex of the $Y$ partite set in the anticlockwise direction from $x_i$, that is adjacent to $x_i$. Note that $W_i=\phi$, if there exists no such vertex of $Y$. The sequence continues leftward (and wraps around if possible) until a 0 is encountered.\\ 
Similarly, consider the column of $A$, corresponding to the vertex $y_j$. We define the set $W_j$ as the collection of $1$’s in the $y_j$-th column that appear consecutively, starting from the row $x_{m_j}$, where $x_{m_j}$ is the first vertex of the $X$ partite set in the anticlockwise direction from $y_j$, that is adjacent to $y_j$. Note that $W_j=\phi$, if there exists no such vertex of $X$. The sequence continues upward (and wraps around if possible) until a 0 is encountered.\\ 
Then such an ordering of the vertices $X\cup Y$ of $B$ is called a \textit{bi-circular ordering} if the sets $W_i$'s and $W_j$'s collectively contain all the $1$’s of the biadjacency matrix $A$.

\begin{theo}\label{t4}
    \textit{A bigraph $B=(X,Y,E)$  is a circular-arc bigraph if and only if there exist a bi-circular ordering of the vertices of $B$.} 
\end{theo}
\begin{proof}
    Necessity: Let $B=(X,Y,E)$ be a circular-arc bigraph. Then there exists a circular arc model $\{A_v: v\in X\cup Y\}$, such that $xy\in E$ if and only if $A_x\cap A_y\neq\phi$ for all $x\in X$ and $y\in Y$. Without loss of generality, we may assume that the circular arc model is chosen so that (1) none of its arcs equals to the whole circle, (2) the arcs are closed (i.e., they contain their endpoints), and (3) no two arcs have a common clockwise end point. \\
    Now label the vertices of $B$ from $1$ to $n$ according to increasing order of their clockwise end points, and arrange the rows and columns of the biadjacency matrix of $B$ according to the increasing order of their indices of the vertices of $X$ and $Y$ partite sets respectively (for example, see Figure 3). We claim that with this arrangement of the biadjacency matrix, $W_i$'s contain all the 1's.
    \par Let the \((x_i, y_j)\) position of the biadjacency matrix contain a \(1\), which implies that $x_i$ is adjacent to $y_j$. If \(i > j\), then based on the ordering of the vertices of \(B\), one of the following must hold:
\begin{itemize}
    \item  \(x_i y_k \in E\) for all possible \(k \in \{j + 1, j + 2, \dots, i - 1\}\). In this case, the \(1\) at position \((x_i, y_j)\) in the biadjacency matrix of \(B\) must be contained in \(W_i\).

\item \(x_l y_j \in E\) for all possible \(l \in \{i + 1, i + 2, \dots, n, 1, \dots, j - 1\}\). In this case, the \(1\) at position \((x_i, y_j)\) in the biadjacency matrix of \(B\) must be contained in \(W_j\).
\end{itemize}
Similarly, if \(i < j\), a parallel argument shows that the \(1\) at position \((x_i, y_j)\) must be contained in either \(W_i\) or \(W_j\).

Thus, in every case, either \(W_i\) or \(W_j\) must contain the \(1\) at position \((x_i, y_j)\) in the biadjacency matrix of the bigraph \(B\). Therefore, the sets \(W_i\) \((1 \leq i \leq n)\) collectively contain all the \(1\)'s in the biadjacency matrix.

    \par Sufficiency: 
Consider a bigraph $B=(X,Y,E)$, where the vertex set $X\cup Y$ of $B$ is ordered as $v_1$, $v_2$,...,$v_n$. Next we place the vertices on an $n$-hour clock, where the $i$-th vertex is placed on the $i$-th hour marker, and assume this ordering ensures that $W_i$'s contain all the $1$'s of the biadjacency matrix of $B$.\\
 If the $i$-th vertex of the ordering is $x_i$, then it will represent a row of the biadjacency matrix and similarly, if $j$-th vertex is $y_j$ it will represent a column.  \\
Let the set $W_i$ starts from the position $(x_i,y_{m_i})$ of the biadjacency matrix of $B$ and continues leftward (and around if possible until a zero is encountered), extending up to the position $(x_i,y_{p_i})$. Then, draw an arc $A_i$ in the clockwise direction starting from $p_i$ and extending up to $i$ on the n-hour clock and, associate this arc with the vertex $x_i$.\\
Let the set $W_j$ starts from the position $(x_{m_j},y_{j})$ of the biadjacency matrix of $B$ and continues upward (and around if possible until a zero is encountered), extending up to the position $(x_{q_j},y_{j})$. Then, draw an arc $A_j$ in the clockwise direction starting from $q_j$ and extending up to $j$ on the n-hour clock and, associate this arc with the vertex $y_j$.\\
Let us assume that  $x_iy_j\in E$. Then, in the biadjacency matrix, the position  $(x_i,y_j)$ is a $1$. Consequently, either $W_i$ or $W_j$ contains this $1$. If $W_i$ contains this $1$, then the arc $A_i$ contains $j$, Alternatively,  if $ W_j$ contains this $1$, then the arc $A_j$ contains $i$. Therefore, in either case, $A_i\cap A_j\neq \phi$.\\
Conversely, let  $A_i\cap A_j\neq \phi$. Then, either the arc $A_i$ contains the clockwise end point of $A_j$,  or the arc $A_j$ contains the clockwise end point of $A_i$. This means either $A_i$ contains $j$ or $A_j$ contains $i$.\\
If $A_i$ contains $j$, then $W_i$  will include the position 
$(x_i,y_j)$ of the biadjacency matrix. Consequently, the position $(x_i,y_j)$ of the biadjacency matrix will be 1, implying that $x_iy_j\in E$.\\
Similarly, if $A_j$ contains $i$, then $W_j$  will include the position 
$(x_i,y_j)$ in the biadjacency matrix of $B$. Consequently, the position $(x_i,y_j)$ in the biadjacency matrix will be 1, implying that $x_iy_j\in E$. Hence, $x_i$ is adjacent to $y_j$ if and only if $A_i\cap A_j\neq\phi$.\\
Therefore,  $B=(X,Y,E)$ is a circular-arc bigraph.
\end{proof}

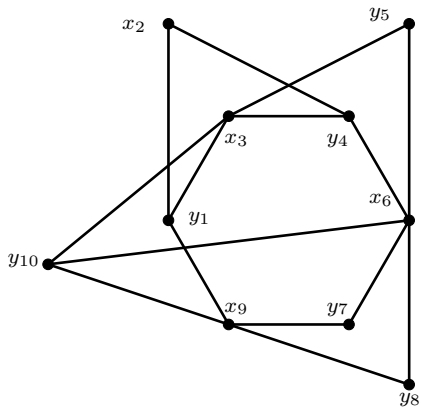
\begin{figure}[H]
    \centering
    \begin{tikzpicture}[line cap=round,line join=round,x=1.0cm,y=1.0cm,scale=.8]
\clip(2.,1.) rectangle (10.,9.);
\fill[line width=1.pt,color=black,fill=white,fill opacity=0.10000000149011612] (6.,3.) -- (8.,3.) -- (9.,4.732050807568878) -- (8.,6.464101615137755) -- (6.,6.464101615137755) -- (5.,4.7320508075688785) -- cycle;
\draw [line width=1.pt,color=black] (6.,3.)-- (8.,3.);
\draw [line width=1.pt,color=black] (8.,3.)-- (9.,4.732050807568878);
\draw [line width=1.pt,color=black] (9.,4.732050807568878)-- (8.,6.464101615137755);
\draw [line width=1.pt,color=black] (8.,6.464101615137755)-- (6.,6.464101615137755);
\draw [line width=1.pt,color=black] (6.,6.464101615137755)-- (5.,4.7320508075688785);
\draw [line width=1.pt,color=black] (5.,4.7320508075688785)-- (6.,3.);
\draw [line width=1.pt] (9.,8.)-- (6.,6.464101615137755);
\draw [line width=1.pt] (9.,8.)-- (9.,4.732050807568878);
\draw [line width=1.pt] (6.,3.)-- (9.,2.);
\draw [line width=1.pt] (9.,4.732050807568878)-- (9.,2.);
\draw [line width=1.pt] (6.,6.464101615137755)-- (3.,4.);
\draw [line width=1.pt] (3.,4.)-- (6.,3.);
\draw [line width=1.pt] (3.,4.)-- (9.,4.732050807568878);
\draw [line width=1.pt] (5.,4.7320508075688785)-- (5.,8.);
\draw [line width=1.pt] (5.,8.)-- (8.,6.464101615137755);
\begin{scriptsize}
\draw [fill=black] (6.,3.) circle (2.5pt);
\draw [fill=black] (8.,3.) circle (2.5pt);
\draw [fill=black] (9.,4.732050807568878) circle (2.5pt);
\draw [fill=black] (8.,6.464101615137755) circle (2.5pt);
\draw [fill=black] (6.,6.464101615137755) circle (2.5pt);
\draw [fill=black] (5.,4.7320508075688785) circle (2.5pt);
\draw [fill=black] (9.,8.) circle (2.5pt);
\draw [fill=black] (9.,2.) circle (2.5pt);
\draw[fill=black] (3.,4.) circle (2.5pt);
\draw [fill=black] (5.,8.) circle (2.5pt);

\draw (4.1,8.2) node[anchor=north west] {$x_2$};
\draw (8.2,8.4) node[anchor=north west] {$y_5$};
\draw (5.2,5) node[anchor=north west] {$y_1$};
\draw (5.8,6.3) node[anchor=north west] {$x_3$};
\draw (7.5,6.3) node[anchor=north west] {$y_4$};
\draw (8.2,5.3) node[anchor=north west] {$x_6$};
\draw (7.5,3.5) node[anchor=north west] {$y_7$};
\draw (5.8,3.5) node[anchor=north west] {$x_9$};
\draw (2.2,4.3) node[anchor=north west] {$y_{10}$};
\draw (8.7,2.) node[anchor=north west] {$y_8$};

\end{scriptsize}
\end{tikzpicture}
    \caption{A bigraph having an ordering of its vertices: $y_1$, $x_2$, $x_3$, $y_4$, $y_5$, $x_6$, $y_7$, $y_8$, $x_9$, $y_{10}$. It is a bi-circular ordering as shown in the next figure.}
    \label{fig:enter-label 10}
\end{figure}

\begin{figure}[H]
    \centering

    \begin{tikzpicture}[line cap=round,line join=round,x=1cm,y=1cm,scale=1.5]
\clip(1.5,1.5) rectangle (8.1,6.5);
\draw [line width=1.pt] (2.,6.)-- (8.,6.);
\draw [line width=1.pt] (2.,6.)-- (2.,2.);
\draw [line width=1.pt] (2.,2.)-- (8.,2.);
\draw [line width=1.pt] (8.,6.)-- (8.,2.);
\draw [line width=1.pt] (3.,6.)-- (3.,2.);
\draw [line width=1.pt] (4.,6.)-- (4.,2.);
\draw [line width=1.pt] (5.,6.)-- (5.,2.);
\draw [line width=1.pt] (6.,6.)-- (6.,2.);
\draw [line width=1.pt] (7.,6.)-- (7.,2.);
\draw [line width=1.pt] (2.,5.)-- (8.,5.);
\draw [line width=1.pt] (2.,4.)-- (8.,4.);
\draw [line width=1.pt] (2.,3.)-- (8.,3.);
\draw (2.3,5.8) node[anchor=north west] {$1$};
\draw (3.3,5.8) node[anchor=north west] {$1$};
\draw (4.3,5.8) node[anchor=north west] {$0$};
\draw (5.3,5.8) node[anchor=north west] {$0$};
\draw (6.3,5.8) node[anchor=north west] {$0$};
\draw (7.2,5.8) node[anchor=north west] {$0$};
\draw (2.3,4.7) node[anchor=north west] {$1$};
\draw (3.3,4.8) node[anchor=north west] {$1$};
\draw (4.3,4.8) node[anchor=north west] {$1$};
\draw (5.3,4.8) node[anchor=north west] {$0$};
\draw (6.3,4.8) node[anchor=north west] {$0$};
\draw (7.2,4.6) node[anchor=north west] {$1$};
\draw (2.3,3.8) node[anchor=north west] {$0$};
\draw (3.3,3.8) node[anchor=north west] {$1$};
\draw (4.3,3.8) node[anchor=north west] {$1$};
\draw (5.1,3.8) node[anchor=north west] {$1$};
\draw (6.2,3.8) node[anchor=north west] {$1$};
\draw (7.2,3.8) node[anchor=north west] {$1$};
\draw (2.3,2.8) node[anchor=north west] {$1$};
\draw (3.3,2.8) node[anchor=north west] {$0$};
\draw (4.3,2.8) node[anchor=north west] {$0$};
\draw (5.3,2.8) node[anchor=north west] {$1$};
\draw (6.2,2.8) node[anchor=north west] {$1$};
\draw (7.2,2.8) node[anchor=north west] {$1$};
\draw (2.3,6.5) node[anchor=north west] {$y_1$};
\draw (3.3,6.5) node[anchor=north west] {$y_4$};
\draw (4.3,6.5) node[anchor=north west] {$y_5$};
\draw (5.3,6.5) node[anchor=north west] {$y_7$};
\draw (6.3,6.5) node[anchor=north west] {$y_8$};
\draw (7.2,6.5) node[anchor=north west] {$y_{10}$};
\draw (1.5,5.6) node[anchor=north west] {$x_2$};
\draw (1.5,4.6) node[anchor=north west] {$x_3$};
\draw (1.5,3.6) node[anchor=north west] {$x_6$};
\draw (1.5,2.6) node[anchor=north west] {$x_9$};

\draw (2.4,5.3) node[anchor=north west,scale=.8] {$W_2$};
\draw (2.2,6.1) node[anchor=north west,scale=2] {$\leftarrow$};

\draw (3.5,4.4) node[anchor=north west,scale=.8] {$W_4$};
\draw (3.5,5) node[anchor=north west,scale=2] {$\uparrow$};
\draw (4.5,3.4) node[anchor=north west,scale=.8] {$W_6$};
\draw (3.8,3.5) node[anchor=north west,scale=2] {$\leftarrow$};
\draw (4.5,5) node[anchor=north west,scale=2] {$\uparrow$};
\draw (4.6,4.4) node[anchor=north west,scale=.8] {$W_5$};
\draw (5.5,3.5) node[anchor=north west,scale=.8] {$W_7$};
\draw (5.5,4.2) node[anchor=north west,scale=2] {$\uparrow$};
\draw (6.5,3.5) node[anchor=north west,scale=.8] {$W_8$};
\draw (6.5,4.2) node[anchor=north west,scale=2] {$\uparrow$};
\draw (7.5,2.5) node[anchor=north west,scale=.8] {$W_{10}$};
\draw (7.5,3.2) node[anchor=north west,scale=2] {$\uparrow$};
\draw (6.5,2.5) node[anchor=north west,scale=.8] {$W_9$};
\draw (6.,2.4) node[anchor=north west,scale=2] {$\leftarrow$};
\draw (2.6,2.5) node[anchor=north west,scale=.8] {$W_1$};
\draw (2.6,4.3) node[anchor=north west,scale=.8] {$W_3$};
\draw (7.6,4.5) node[anchor=north west,scale=.8] {$W_3$};
\draw (7.5,5.1) node[anchor=north west,scale=2] {$\leftarrow$};
\draw (2.,5.1) node[anchor=north west,scale=2] {$\leftarrow$};
\draw (2.5,3.2) node[anchor=north west,scale=2] {$\uparrow$};

\draw [rotate around={3.6926304246035833:(2.4952513105674035,5.600415464175756)},line width=1.pt] (2.4952513105674035,5.600415464175756) ellipse (0.45989336694701877cm and 0.17910190566759865cm);
\draw [rotate around={88.65021006226398:(3.4641754996789644,5.087458383845753)},line width=1.pt] (3.4641754996789644,5.087458383845753) ellipse (0.8544924032491376cm and 0.2638682151235102cm);
\draw [rotate around={-84.59603096336019:(4.434560058552211,4.479870230574548)},line width=1.pt] (4.434560058552211,4.479870230574548) ellipse (0.36985828967970974cm and 0.18475291725939152cm);
\draw [rotate around={0.3670793417873507:(3.8599342672833945,3.588604687970298)},line width=1.pt] (3.8599342672833945,3.588604687970298) ellipse (0.9814755214657788cm and 0.29998671025002643cm);
\draw [rotate around={-79.72449639411053:(5.331306818138043,3.5300876178691634)},line width=1.pt] (5.331306818138043,3.5300876178691634) ellipse (0.36786574989662096cm and 0.1813151432771058cm);
\draw [rotate around={-78.48302149308454:(6.372475636264346,3.5391230185683065)},line width=1.pt] (6.372475636264346,3.5391230185683065) ellipse (0.35609697816325675cm and 0.15904292921167085cm);
\draw [rotate around={-76.24629305081747:(2.4339917705901803,2.556635943100596)},line width=1.pt] (2.4339917705901803,2.556635943100596) ellipse (0.3274502313394834cm and 0.1908896504224311cm);
\draw [rotate around={-1.703603336971033:(5.879044750448265,2.5694693459677054)},line width=1.pt] (5.879044750448265,2.5694693459677054) ellipse (0.8825329803410645cm and 0.2646622120096985cm);
\draw [rotate around={-88.24799907488993:(7.346262176392747,3.6064930042888994)},line width=1.pt] (7.346262176392747,3.6064930042888994) ellipse (1.3007848457060887cm and 0.2490532656120195cm);
\draw [rotate around={2.334953890183946:(8.0588258421206,4.455395969862581)},line width=1.pt] (8.0588258421206,4.455395969862581) ellipse (1.2986416254752358cm and 0.3214905846074233cm);
\draw [rotate around={0.9557274781593185:(1.6456091685234757,4.458393043263531)},line width=1.pt] (1.6456091685234757,4.458393043263531) ellipse (1.2423112706262163cm and 0.3399681860271004cm);
\end{tikzpicture}
    \caption{The biadjacency matrix of the bigraph in Figure 2, where the rows and columns are arranged according to the increasing order of their indices, and corresponding $W_i$'s and $W_j$'s.}
    \label{fig:enter-label 11}
\end{figure}
\noindent{}The circular-arc representation of the bigraph of Figure 2 is the following: $A_{y_1}=[9,1]$, $A_{x_2}=[1,2]$, $A_{x_3}=[10,3]$, $A_{y_4}=[2,4]$, $A_{y_5}=[3,5]$, $A_{x_6}=[4,6]$, $A_{y_7}=[6,7]$, $A_{y_8}=[6,8]$, $A_{x_9}=[7,9]$ and, $A_{y_{10}}=[3,10]$.\\

 \par Pavol Hell and Jing Huang \cite{hell} characterized interval bigraphs using forbidden patterns with respect to a specific vertex ordering in the following theorem.
\begin{theo}[\cite{hell}]
   \textit{ Let $H$ be a bipartite graph with bipartition $(X,Y)$. Then the following statements are equivalent}:
    \begin{itemize}
        \item\textit{ $H$ is an interval bigraph};
        \item\textit{ the vertices of $H$ can be ordered $v_1$, $v_2$, ..., $v_n$, so that there do not exist $a<b<c$ in the configurations in Figure 4. (Black vertices are in $X$, red vertices in $Y$, or conversely, and all edges not shown are absent.)}
        \item \textit{the vertices of $H$ can be ordered $v_1$, $v_2$, ..., $v_n$, so that there do not exist $a<b<c<d$ in the configurations in Figure 5}.
    \end{itemize}
\end{theo}
\begin{figure}[H]
    \centering
    \begin{tikzpicture}[line cap=round,line join=round,x=1.0cm,y=1.0cm,scale=.8]
\clip(1.,-1) rectangle (7.,3.);
\draw [shift={(4.,-1.)},line width=1.pt]  plot[domain=0.7853981633974483:2.356194490192345,variable=\t]({1.*2.8284271247461903*cos(\t r)+0.*2.8284271247461903*sin(\t r)},{0.*2.8284271247461903*cos(\t r)+1.*2.8284271247461903*sin(\t r)});
\begin{scriptsize}
\draw [fill=black] (2.,1.) circle (3.5pt);
\draw [fill=black] (4.,1.) circle (3.5pt);
\draw [fill=red] (6.,1.) circle (3.5pt);
\draw (1.5,.8) node[anchor=north west,scale=1.5] {$v_a$};
\draw (3.5,.8) node[anchor=north west,scale=1.5] {$v_b$};
\draw (5.5,.8) node[anchor=north west,scale=1.5] {$v_c$};
\end{scriptsize}
\end{tikzpicture}
    \caption{Forbidden pattern.}
    \label{fig:enter-label-3}
\end{figure}
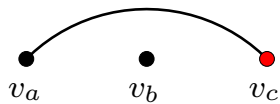
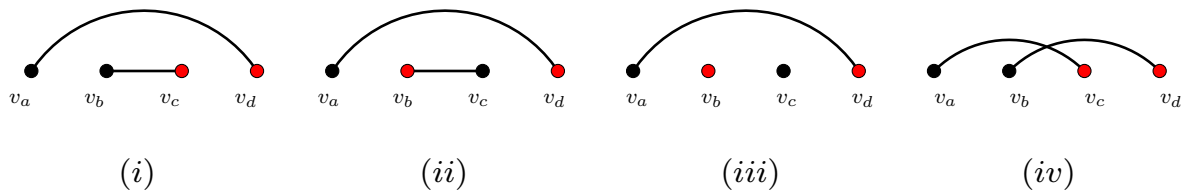
\begin{figure}[H]
    \centering
    \begin{tikzpicture}[line cap=round,line join=round,x=1.0cm,y=1.0cm, scale=1]
\clip(1.,0.) rectangle (19.,4.);
\draw [line width=1.pt] (3.,2.)-- (4.,2.);
\draw [shift={(3.5,1.)},line width=1.pt]  plot[domain=0.5880026035475675:2.5535900500422257,variable=\t]({1.*1.8027756377319948*cos(\t r)+0.*1.8027756377319948*sin(\t r)},{0.*1.8027756377319948*cos(\t r)+1.*1.8027756377319948*sin(\t r)});
\draw [line width=1.pt] (7.,2.)-- (8.,2.);
\draw [shift={(7.5,1.)},line width=1.pt]  plot[domain=0.5880026035475675:2.5535900500422257,variable=\t]({1.*1.8027756377319948*cos(\t r)+0.*1.8027756377319948*sin(\t r)},{0.*1.8027756377319948*cos(\t r)+1.*1.8027756377319948*sin(\t r)});
\draw [shift={(11.5,1.)},line width=1.pt]  plot[domain=0.5880026035475675:2.5535900500422257,variable=\t]({1.*1.8027756377319948*cos(\t r)+0.*1.8027756377319948*sin(\t r)},{0.*1.8027756377319948*cos(\t r)+1.*1.8027756377319948*sin(\t r)});
\draw [shift={(15.,1.)},line width=1.pt]  plot[domain=0.7853981633974483:2.356194490192345,variable=\t]({1.*1.4142135623730951*cos(\t r)+0.*1.4142135623730951*sin(\t r)},{0.*1.4142135623730951*cos(\t r)+1.*1.4142135623730951*sin(\t r)});
\draw [shift={(16.,1.)},line width=1.pt]  plot[domain=0.7853981633974483:2.356194490192345,variable=\t]({1.*1.4142135623730951*cos(\t r)+0.*1.4142135623730951*sin(\t r)},{0.*1.4142135623730951*cos(\t r)+1.*1.4142135623730951*sin(\t r)});
\begin{scriptsize}
\draw [fill=black] (2.,2.) circle (2.5pt);
\draw [fill=black] (3.,2.) circle (2.5pt);
\draw [fill=red] (4.,2.) circle (2.5pt);
\draw [fill=red] (5.,2.) circle (2.5pt);
\draw [fill=black] (6.,2.) circle (2.5pt);
\draw [fill=red] (7.,2.) circle (2.5pt);
\draw [fill=black] (8.,2.) circle (2.5pt);
\draw [fill=red] (9.,2.) circle (2.5pt);
\draw [fill=black] (10.,2.) circle (2.5pt);
\draw [fill=red] (11.,2.) circle (2.5pt);
\draw [fill=black] (12.,2.) circle (2.5pt);
\draw [fill=red] (13.,2.) circle (2.5pt);
\draw [fill=black] (14.,2.) circle (2.5pt);
\draw [fill=black] (15.,2.) circle (2.5pt);
\draw [fill=red] (16.,2.) circle (2.5pt);
\draw [fill=red] (17.,2.) circle (2.5pt);
\draw (1.6,1.8) node[anchor=north west,scale=1.] {$v_a$};
\draw (2.6,1.8) node[anchor=north west,scale=1.] {$v_b$};
\draw (3.6,1.8) node[anchor=north west,scale=1.] {$v_c$};
\draw (4.6,1.8) node[anchor=north west,scale=1.] {$v_d$};
\draw (3,1) node[anchor=north west,scale=1.5] {$(i)$};

\draw (5.7,1.8) node[anchor=north west,scale=1.] {$v_a$};
\draw (6.7,1.8) node[anchor=north west,scale=1.] {$v_b$};
\draw (7.7,1.8) node[anchor=north west,scale=1.] {$v_c$};
\draw (8.7,1.8) node[anchor=north west,scale=1.] {$v_d$};
\draw (7,1) node[anchor=north west,scale=1.5] {$(ii)$};

\draw (9.8,1.8) node[anchor=north west,scale=1.] {$v_a$};
\draw (10.8,1.8) node[anchor=north west,scale=1.] {$v_b$};
\draw (11.8,1.8) node[anchor=north west,scale=1.] {$v_c$};
\draw (12.8,1.8) node[anchor=north west,scale=1.] {$v_d$};
\draw (11,1) node[anchor=north west,scale=1.5] {$(iii)$};

\draw (13.9,1.8) node[anchor=north west,scale=1.] {$v_a$};
\draw (14.9,1.8) node[anchor=north west,scale=1.] {$v_b$};
\draw (15.9,1.8) node[anchor=north west,scale=1.] {$v_c$};
\draw (16.9,1.8) node[anchor=north west,scale=1.] {$v_d$};
\draw (15,1) node[anchor=north west,scale=1.5] {$(iv)$};

\end{scriptsize}
\end{tikzpicture}
    \caption{Forbidden patterns.}
    \label{fig:enter-label-4}
\end{figure}
Motivated by the above result of Hell and Huang \cite{hell}, in the Theorem 4 we provide a characterization of circular-arc bigraphs in terms of forbidden patterns.
\begin{theo}
   \textit{ Let $G$ be a bipartite graph with bipartition $(X,Y)$. Then the following statements are equivalent}:
    \begin{itemize}
        \item \textit{$G$ is a circular-arc bigraph;} 
        \item \textit{The vertices of $G$ can be ordered $v_1$, $v_2$, $v_3$,..., $v_n$, so that there do not exist $i<j<k<l$  in the configurations in Figure 6}. (Black vertices are in $X$, red vertices in $Y$, or conversely, and all edges not shown are absent.)
    \end{itemize}
\end{theo}
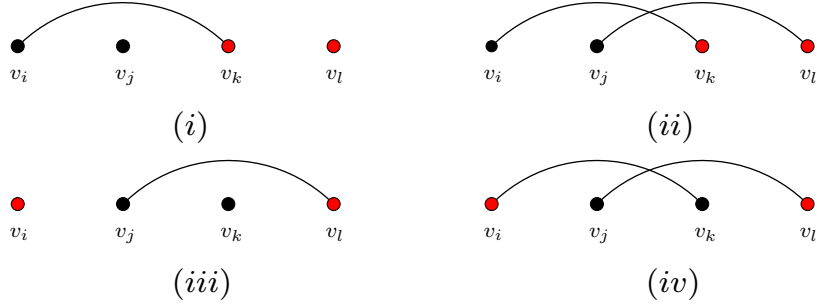
\begin{figure}[H]
    \centering
    \begin{tikzpicture}[line cap=round,line join=round,x=1.0cm,y=1.0cm,scale=.7]
\clip(1.,-1) rectangle (18.,6.);
\draw [shift={(4.,2.)},line width=.5pt]  plot[domain=0.7853981633974483:2.356194490192345,variable=\t]({1.*2.8284271247461903*cos(\t r)+0.*2.8284271247461903*sin(\t r)},{0.*2.8284271247461903*cos(\t r)+1.*2.8284271247461903*sin(\t r)});
\draw [shift={(13.,2.)},line width=.5pt]  plot[domain=0.7853981633974483:2.356194490192345,variable=\t]({1.*2.8284271247461903*cos(\t r)+0.*2.8284271247461903*sin(\t r)},{0.*2.8284271247461903*cos(\t r)+1.*2.8284271247461903*sin(\t r)});
\draw [shift={(15.,2.)},line width=.5pt]  plot[domain=0.7853981633974483:2.356194490192345,variable=\t]({1.*2.8284271247461903*cos(\t r)+0.*2.8284271247461903*sin(\t r)},{0.*2.8284271247461903*cos(\t r)+1.*2.8284271247461903*sin(\t r)});
\draw [shift={(6.,-1.)},line width=.5pt]  plot[domain=0.7853981633974483:2.356194490192345,variable=\t]({1.*2.8284271247461903*cos(\t r)+0.*2.8284271247461903*sin(\t r)},{0.*2.8284271247461903*cos(\t r)+1.*2.8284271247461903*sin(\t r)});
\draw [shift={(13.,-1.)},line width=.5pt]  plot[domain=0.7853981633974483:2.356194490192345,variable=\t]({1.*2.8284271247461903*cos(\t r)+0.*2.8284271247461903*sin(\t r)},{0.*2.8284271247461903*cos(\t r)+1.*2.8284271247461903*sin(\t r)});
\draw [shift={(15.,-1.)},line width=.5pt]  plot[domain=0.7853981633974483:2.356194490192345,variable=\t]({1.*2.8284271247461903*cos(\t r)+0.*2.8284271247461903*sin(\t r)},{0.*2.8284271247461903*cos(\t r)+1.*2.8284271247461903*sin(\t r)});
\begin{scriptsize}
\draw [fill=black] (2.,4.) circle (3.5pt);
\draw (1.7,3.7) node[anchor=north west,scale=1] {$v_i$};
\draw [fill=black] (4.,4.) circle (3.5pt);
\draw (3.7,3.7) node[anchor=north west,scale=1] {$v_j$};
\draw [fill=red] (6.,4.) circle (3.5pt);
\draw (5.7,3.7) node[anchor=north west,scale=1] {$v_k$};
\draw [fill=red] (8.,4.) circle (3.5pt);
\draw (7.7,3.7) node[anchor=north west,scale=1] {$v_l$};
\draw (4.7,3) node[anchor=north west,scale=1.5] {$(i)$};
\draw [fill=black] (11.,4.) circle (3pt);
\draw (10.7,3.7) node[anchor=north west,scale=1] {$v_i$};
\draw [fill=black] (13.,4.) circle (3.5pt);
\draw (12.7,3.7) node[anchor=north west,scale=1] {$v_j$};
\draw [fill=red] (15.,4.) circle (3.5pt);
\draw (14.7,3.7) node[anchor=north west,scale=1] {$v_k$};
\draw [fill=red] (17.,4.) circle (3.5pt);
\draw (16.7,3.7) node[anchor=north west,scale=1] {$v_l$};
\draw [fill=red] (2.,1.) circle (3.5pt);
\draw (13.7,3) node[anchor=north west,scale=1.5] {$
(ii)$};
\draw (1.7,.7) node[anchor=north west,scale=1] {$v_i$};
\draw [fill=black] (4.,1.) circle (3.5pt);
\draw (3.7,.7) node[anchor=north west,scale=1] {$v_j$};
\draw [fill=black] (6.,1.) circle (3.5pt);
\draw (5.7,.7) node[anchor=north west,scale=1] {$v_k$};
\draw [fill=red] (8.,1.) circle (3.5pt);
\draw (7.7,.7) node[anchor=north west,scale=1] {$v_l$};
\draw [fill=red] (11.,1.) circle (3.5pt);
\draw (4.7,0) node[anchor=north west,scale=1.5] {$
(iii)$};
\draw (10.7,.7) node[anchor=north west,scale=1] {$v_i$};
\draw [fill=black] (13.,1.) circle (3.5pt);
\draw (12.7,.7) node[anchor=north west,scale=1] {$v_j$};
\draw [fill=black] (15.,1.) circle (3.5pt);
\draw (14.7,.7) node[anchor=north west,scale=1] {$v_k$};
\draw [fill=red] (17.,1.) circle (3.5pt);
\draw (16.7,.7) node[anchor=north west,scale=1] {$v_l$};
\draw (13.7,0) node[anchor=north west,scale=1.5] {$
(iv)$};
\end{scriptsize}
\end{tikzpicture}
    \caption{Forbidden patterns.}
    \label{fig:enter-label-5}
\end{figure}

\begin{proof}
    Necessity: Let $B=(X,Y,E)$  be a circular-arc bigraph with 
$n$ vertices. Arrange the vertices of $B$ according to  increasing order of the clockwise endpoints of the corresponding circular arcs. Let this ordering of the vertices be $v_1$, $v_2$, $v_3$,..., $v_n$. \\
   Now, we need to demonstrate that in this specific ordering, the configurations shown in Figure 6 cannot occur.\\
   Let $v_i$, $v_j$, $v_k$, $v_l$ be four vertices of $B$ such that $i<j<k<l$, where $v_i$, $v_j$ $\in$ $X$ and $v_k$, $v_l$ $\in$ $Y$. If $v_i$ is adjacent to $v_k$, it leads to the following two possible cases (where white vertices can belong to any partite set):\\
   
   \textbf{Case 1.}
   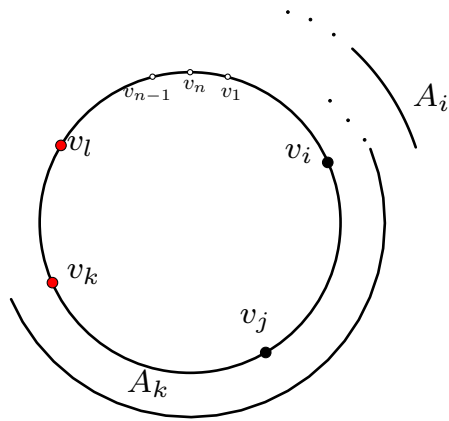
\begin{figure}[H]
       \centering
       \begin{tikzpicture}[line cap=round,line join=round,x=1.0cm,y=1.0cm,scale=1.]
\clip(6.,3.) rectangle (13.,9.);
\draw [line width=1.pt] (9.,6.) circle (2.cm);
\draw [shift={(9.,6.)},line width=1.pt]  plot[domain=-2.7367008673047097:0.3876695783739986,variable=\t]({1.*2.5893628559937287*cos(\t r)+0.*2.5893628559937287*sin(\t r)},{0.*2.5893628559937287*cos(\t r)+1.*2.5893628559937287*sin(\t r)});
\draw [shift={(9.,6.)},line width=1.pt]  plot[domain=0.32175055439664224:0.8204714939306736,variable=\t]({1.*3.1622776601683795*cos(\t r)+0.*3.1622776601683795*sin(\t r)},{0.*3.1622776601683795*cos(\t r)+1.*3.1622776601683795*sin(\t r)});
\begin{scriptsize}
\draw [fill=white] (9.5,7.94) circle (1.pt);
\draw (9.3,7.9) node[anchor=north west,scale=1] {$v_1$};

\draw [fill=white] (9.,8.) circle (1.pt);
\draw (8.8,8) node[anchor=north west,scale=1] {$v_n$};

\draw [fill=white] (8.5,7.94) circle (1.pt);
\draw (8.,7.9) node[anchor=north west,scale=1] {$v_{n-1}$};

\draw [fill=black] (10.83,6.8) circle (2.pt);
\draw (10.1,7.2) node[anchor=north west,scale=1.5] {$v_i$};

\draw [fill=black] (10.005240914749061,4.270985626631615) circle (2.pt);
\draw (9.5,5) node[anchor=north west,scale=1.5] {$v_j$};

\draw [fill=red] (7.17,5.2) circle (2.pt);
\draw (7.2,5.6) node[anchor=north west,scale=1.5] {$v_k$};

\draw [fill=red] (7.283275503697917,7.026088204686302) circle (2.pt);
\draw (7.2,7.3) node[anchor=north west,scale=1.5] {$v_l$};

\draw [fill=black] (11.32,7.1) circle (.5pt);
\draw [fill=black] (11.1,7.36) circle (.5pt);
\draw [fill=black] (10.86,7.62) circle (.5pt);
\draw [fill=black] (10.3,8.8) circle (.5pt);
\draw [fill=black] (10.6,8.7) circle (.5pt);
\draw [fill=black] (10.91,8.5) circle (.5pt);
\draw (8,4.2) node[anchor=north west,scale=1.5] {$A_k$};
\draw (11.8,8) node[anchor=north west,scale=1.5] {$A_i$};

\end{scriptsize}
\end{tikzpicture}
       \caption{Clockwise end point of $A_i$ lies in $A_k$.}
       \label{fig:enter-label-6}
   \end{figure}
   For this case, $v_j$ must be adjacent to $v_k$. However, in the Figure 6(i), there is no edge between $v_j$ and $v_k$, which contradicts the first configuration of Figure 6.\\
   \par
   \textbf{Case 2.}
   
\begin{figure}[H]
    \centering

   \begin{tikzpicture}[line cap=round,line join=round,x=1.0cm,y=1.0cm,scale=1.]
\clip(5.,1.) rectangle (11.,8.);
\draw [line width=1.pt] (8.,5.) circle (2.cm);
\draw [shift={(8.,5.)},line width=1.pt]  plot[domain=0.3966104021074092:3.53219969728748,variable=\t]({1.*2.536927275268253*cos(\t r)+0.*2.536927275268253*sin(\t r)},{0.*2.536927275268253*cos(\t r)+1.*2.536927275268253*sin(\t r)});
\draw [shift={(8.,5.)},line width=1.pt]  plot[domain=3.4972285378905528:4.71238898038469,variable=\t]({1.*2.986904752415115*cos(\t r)+0.*2.986904752415115*sin(\t r)},{0.*2.986904752415115*cos(\t r)+1.*2.986904752415115*sin(\t r)});
\begin{scriptsize}
\draw [fill=white] (8.,7.) circle (1pt);
\draw [fill=white] (8.827605888602369,6.820732954925209) circle (1pt);
\draw [fill=white] (7.189651391201308,6.828478912161151) circle (1pt);
\draw [fill=black] (9.85,5.8) circle (2.5pt);
\draw [fill=black] (9.714985851425087,3.9710084891449458) circle (2.5pt);
\draw [fill=red] (6.13,4.3) circle (2.5pt);
\draw [fill=red] (6.269596358729952,6.002847564826958) circle (2.5pt);
\draw [fill=black] (5.74,3.92) circle (.5pt);
\draw [fill=black] (5.98,3.68) circle (.5pt);
\draw [fill=black] (6.2,3.46) circle (.5pt);
\draw [fill=black] (8.26,2.08) circle (.5pt);
\draw [fill=black] (8.64,2.16) circle (.5pt);
\draw [fill=black] (8.96,2.28) circle (.5pt);
\draw (7.8,7) node[anchor=north west,scale=1] {$v_n$};
\draw (7.,6.8) node[anchor=north west,scale=1] {$v_{n-1}$};
\draw (8.4,6.8) node[anchor=north west,scale=1] {$v_1$};
\draw (5,3) node[anchor=north west,scale=1.5] {$A_k$};
\draw (8,8.1) node[anchor=north west,scale=1.5] {$A_i$};
\draw (9,6.1) node[anchor=north west,scale=1.5] {$v_i$};
\draw (9,4.3) node[anchor=north west,scale=1.5] {$v_j$};
\draw (6.2,4.4) node[anchor=north west,scale=1.5] {$v_k$};
\draw (6.2,6.1) node[anchor=north west,scale=1.5] {$v_l$};

\end{scriptsize}
\end{tikzpicture}
 \caption{Clockwise end point of $A_k$ lies in $A_i$.}
    \label{fig:enter-label-7}
\end{figure}
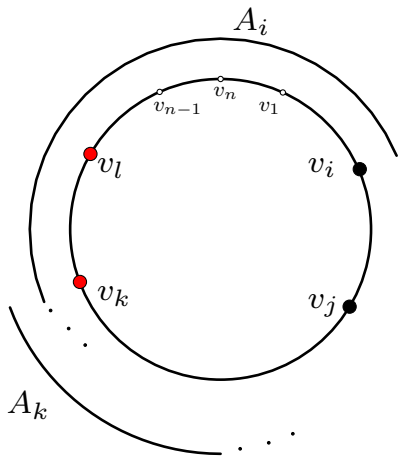
 For this case, $v_l$ must be adjacent to $v_i$. However, in the Figure 6(i), there is no edge between $v_l$ and $v_i$, which contradicts the first configuration of Figure 6.\\
 Therefore there do not exist $i<j<k<l$ in the configuration $(i)$ of Figure 6.\\
 Next, by similar argument, we can examine that for this ordering of the vertices there do not exist $i<j<k<l$ in the configuration $(ii)$, $(iii)$ and $(iv)$ in Figure 6.
 \par Sufficiency: Let us consider that the vertices of $B=(X,Y,E)$ can be ordered $v_1$, $v_2$, ..., $v_n$. So that there do not exist $i<j<k<l$ corresponding to any of the configurations in Figure 6.\\
 We will now construct a family of circular arcs $\mathcal{A}=\{A_i: 1\leq i\leq n\}$ corresponding to the vertices of $B$. Let $A_i=[m_i,i]$, ($1\leq i\leq n$) where $v_{m_i}$  is the last consecutive vertex from the opposite partite set that is adjacent to $v_i$ in the anticlockwise direction (starting from the vertex $v_i$). That is, if $v_i\in X$, then $v_{m_i}\in Y$ and conversely, if $v_i\in Y$, then $v_{m_i}\in X$.\\
 Now, we have only to show that $A_i\cap A_k\neq\phi$ if and only if $v_iv_k\in E$, where $v_i$ and $v_k$ are vertices of different partite sets.\\
 If $A_i\cap A_k\neq\phi$, then this intersection can occur in one of two ways, see Figure $9(i)$ or $9(ii)$.
 \begin{figure}[H]
     \centering
     \begin{tikzpicture}[line cap=round,line join=round,x=1.0cm,y=1.0cm]
\clip(-4.5,1.) rectangle (12,9.);
\draw [line width=1.pt] (7.980034032898469,4.820306296086217) circle (2.1797851460003312cm);
\draw [shift={(7.980034032898469,4.820306296086217)},line width=1.pt]  plot[domain=0.4567461636507309:3.468531325972799,variable=\t]({1.*2.6295111433068192*cos(\t r)+0.*2.6295111433068192*sin(\t r)},{0.*2.6295111433068192*cos(\t r)+1.*2.6295111433068192*sin(\t r)});
\draw [line width=1.pt] (-1.,5.) circle (2.cm);
\draw [shift={(8.,5.)},line width=1.pt]  plot[domain=3.2095853215937065:4.71238898038469,variable=\t]({1.*2.9694753582657367*cos(\t r)+0.*2.9694753582657367*sin(\t r)},{0.*2.9694753582657367*cos(\t r)+1.*2.9694753582657367*sin(\t r)});
\draw [shift={(-1.,5.)},line width=1.pt]  plot[domain=-3.0722065965203265:0.9708947512270215,variable=\t]({1.*2.33397037997512*cos(\t r)+0.*2.33397037997512*sin(\t r)},{0.*2.33397037997512*cos(\t r)+1.*2.33397037997512*sin(\t r)});
\draw [shift={(-1.,5.)},line width=1.pt]  plot[domain=0.499684617825641:1.3763509501398137,variable=\t]({1.*2.579162802238808*cos(\t r)+0.*2.579162802238808*sin(\t r)},{0.*2.579162802238808*cos(\t r)+1.*2.579162802238808*sin(\t r)});
\begin{scriptsize}
\draw [fill=white] (7.980034032898469,7.000091442086548) circle (1pt);
\draw (7.6,7) node[anchor=north west,scale=1.5] {$v_n$};

\draw [fill=black] (9.93,5.8) circle (2.5pt);
\draw [fill=red] (5.800248886898137,4.820306296086217) circle (2.5pt);
\draw [fill=black] (8.26,2.08) circle (.5pt);
\draw [fill=black] (8.64,2.16) circle (.5pt);
\draw [fill=black] (8.96,2.28) circle (.5pt);
\draw [fill=white] (-1.,7.) circle (1pt);
\draw (-1.3,7) node[anchor=north west,scale=1.5] {$v_n$};

\draw [fill=red] (5.966228857665243,3.986008987611857) circle (2.5pt);
\draw [fill=red] (-2.989990136070066,4.800151911833412) circle (2.5pt);
\draw [fill=black] (0.712607201786847,6.032945580554864) circle (2.5pt);
\draw [fill=black] (0.12912247494201323,6.650782371053442) circle (2.5pt);
\draw [fill=black] (-0.7,7.56) circle (.5pt);
\draw [fill=black] (-1.,7.54) circle (.5pt);
\draw [fill=black] (-1.3,7.5) circle (.5pt);
\draw (9.2,6.1) node[anchor=north west,scale=1.5] {$v_i$};
\draw (5.9,5.1) node[anchor=north west,scale=1.5] {$v_k$};
\draw (5.9,4.4) node[anchor=north west,scale=1.5] {$v_{m_i}$};
\draw (.1,6.1) node[anchor=north west,scale=1.5] {$v_i$};
\draw (-.5,6.7) node[anchor=north west,scale=1.5] {$v_{m_k}$};
\draw (-3,5.1) node[anchor=north west,scale=1.5] {$v_k$};
\draw (6.2,7.8) node[anchor=north west,scale=1.5] {$A_i$};
\draw (4.9,3.1) node[anchor=north west,scale=1.5] {$A_k$};
\draw (0.2,7.8) node[anchor=north west,scale=1.5] {$A_i$};
\draw (-2.2,2.8) node[anchor=north west,scale=1.5] {$A_k$};
\draw (-4.6,1.6) node[anchor=north west,scale=1.2] {Figure $9(i)$: clockwise end point of $A_i$ lies in $A_k$.};
\draw (3.8,1.6) node[anchor=north west,scale=1.2] {Figure $9(ii)$: clockwise end point of $A_k$ lies in $A_i$.};

\end{scriptsize}
\end{tikzpicture}
     \label{fig:enter-label-8}
 \end{figure}

Therefore, in either case, by the construction of $A_i$ and $A_k$, it is clear that $v_iv_k\in E$.\\
We now show that if, $v_iv_k\in E$ then $A_i\cap A_k\neq\phi$.
\par If possible, let $A_i\cap A_k=\phi$. By the constructions of $A_i$ and $A_k$, this implies the existence of a vertex $v_j$ from the same partite set as $v_i$ (where $i<j<k$) that is not adjacent to $v_k$ (i.e. $v_kv_j\notin E)$. Additionally, there must be another vertex $v_l$ belongs to the same partite set as $v_k$, has positioned between $v_k$ and $v_i$ in the clockwise direction, which is not adjacent to $v_i$ (i.e. $v_lv_i\notin E)$. Depending on the position of such $v_l$ we have two cases, see Figure $10(i)$ or $10(ii)$.\\

\begin{figure}[H]
    \centering
    \begin{tikzpicture}[line cap=round,line join=round,x=1.0cm,y=1.0cm]
\clip(2,.5) rectangle (18.,8.);
\draw [line width=1.pt] (6.,5.) circle (2.cm);
\draw [shift={(6.,5.)},line width=1.pt]  plot[domain=-0.008333140440135445:2.14717154738608,variable=\t]({1.*2.4000833318866244*cos(\t r)+0.*2.4000833318866244*sin(\t r)},{0.*2.4000833318866244*cos(\t r)+1.*2.4000833318866244*sin(\t r)});
\draw [shift={(6.,5.)},line width=1.pt]  plot[domain=3.141592653589793:5.255715449211019,variable=\t]({1.*2.48*cos(\t r)+0.*2.48*sin(\t r)},{0.*2.48*cos(\t r)+1.*2.48*sin(\t r)});
\draw [line width=1.pt] (13.,5.) circle (2.cm);
\draw [shift={(13.,5.)},line width=1.pt]  plot[domain=3.141592653589793:5.217058339757389,variable=\t]({1.*2.48*cos(\t r)+0.*2.48*sin(\t r)},{0.*2.48*cos(\t r)+1.*2.48*sin(\t r)});
\draw [shift={(13.,5.)},line width=1.pt]  plot[domain=0.:0.8258385310050181,variable=\t]({1.*2.48*cos(\t r)+0.*2.48*sin(\t r)},{0.*2.48*cos(\t r)+1.*2.48*sin(\t r)});
\begin{scriptsize}
\draw [fill=white] (6.,7.) circle (1.5pt);
\draw (5.8,6.8) node[anchor=north west,scale=1] {$v_n$};
\draw [fill=black] (8.,5.) circle (2.5pt);
\draw (7.3,5.2) node[anchor=north west,scale=1.5] {$v_i$};

\draw [fill=black] (7.723868430315539,3.9859597468732124) circle (2.5pt);
\draw (4.1,5.3) node[anchor=north west,scale=1.5] {$v_k$};
\draw (4.3,6.2) node[anchor=north west,scale=1.5] {$v_l$};
\draw (4.9,6.8) node[anchor=north west,scale=1.5] {$v_{m_i}$};
\draw (4.1,3) node[anchor=north west,scale=1.5] {$A_k$};
\draw (6.1,7.9) node[anchor=north west,scale=1.5] {$A_i$};

\draw (7.,4.4) node[anchor=north west,scale=1.5] {$v_j$};

\draw [fill=black] (7.0201530344336085,3.279741941935484) circle (2.5pt);
\draw (6.2,3.9) node[anchor=north west,scale=1.5] {$v_{m_k}$};
\draw [fill=red] (4.,5.) circle (2.5pt);

\draw [fill=red] (4.257384112780107,5.981473315790515) circle (2.5pt);
\draw [fill=red] (4.971008489144947,6.714985851425089) circle (2.5pt);
\draw [fill=white] (13.,7.) circle (1.5pt);
\draw [fill=black] (15.,5.) circle (2.5pt);
\draw [fill=red] (14.329997221961367,6.493689187741227) circle (2.5pt);
\draw [fill=red] (13.742781352708207,6.856953381770519) circle (2.5pt);
\draw [fill=red] (11.,5.) circle (2.5pt);
\draw [fill=black] (13.972971344926611,3.2526228907440444) circle (2.5pt);
\draw [fill=black] (14.734027887970035,4.003432248293083) circle (2.5pt);
\draw (14.3,5.3) node[anchor=north west,scale=1.5] {$v_i$};
\draw (11,5.3) node[anchor=north west,scale=1.5] {$v_k$};
\draw (13.9,4.4) node[anchor=north west,scale=1.5] {$v_j$};
\draw (13.2,3.9) node[anchor=north west,scale=1.5] {$v_{m_k}$};
\draw (13.3,7.47) node[anchor=north west,scale=1.5] {$v_l$};
\draw (4.1,5.3) node[anchor=north west,scale=1.5] {$v_k$};
\draw (13.5,6.6) node[anchor=north west,scale=1.5] {$v_{m_i}$};
\draw (12.7,7.5) node[anchor=north west,scale=1.] {$v_n$};
\draw (10.8,3.) node[anchor=north west,scale=1.5] {$A_k$};
\draw (15.25,6) node[anchor=north west,scale=1.5] {$A_i$};
\draw (2.5,2.2) node[anchor=north west,scale=1.2] {Figure $10(i)$: The vertex $v_l$ is such that};
\draw (5.5,1.8) node[anchor=north west,scale=1.2] { $(k<l\leq n)$.};

\draw (10.2,2.2) node[anchor=north west,scale=1.2] {Figure $10(ii)$: The vertex $v_l$ is such that};
\draw (13.5,1.8) node[anchor=north west,scale=1.2] {$(1\leq l<i)$.};

\end{scriptsize}
\end{tikzpicture}
    \label{fig: enter-label-9}
\end{figure}
If $k<l$, then we have four vertices $v_i$, $v_j$, $v_k$, $v_l$ with $i<j<k<l$ forming one of the configurations shown in Figure 6, either $(i)$ or $(ii)$, depending upon $v_jv_l\notin E$ or $v_jv_l\in E$. In either case we have a contradiction.\\
If $l<i$, then we have four vertices  $v_l$, $v_i$, $v_j$, $v_k$ with $l<i<j<k$ forming one of the configurations shown in Figure 6, either $(iii)$ or $(iv)$, depending upon $v_jv_l\notin E$ or $v_jv_l\in E$ (to simplify, we can rename the indices $l$, $i$, $j$, $k$  as $i$, $j$, $k$, $l$  respectively, which results in one of the configurations either $(iii)$ or $(iv)$ of Figure 6). Again we have a contradiction. Hence $v_iv_k\in E$ $\implies$ $A_i\cap A_k\neq\phi$.\\ 
Therefore, $v_iv_k\in E$ if and only if $A_i\cap A_k\neq\phi$ and so $B=(X,Y,E)$ is a circular-arc bigraph.

\end{proof}
\section{Conclusion}
The recognition algorithm of circular-arc graphs have been found in linear time after a long research \cite{dhh,McConnel}. As mentioned  before that very recently Francis, Hell, and Stacho \cite{fhs} have found certifying recognition algorithm for circular-arc graphs with running time $\mathcal{O} (n^3)$. Their algorithm is based on forbidden structures of circular-arc graphs. But the problem of finding an efficient recognition algorithm of circular-arc bigraphs is still open. We do hope that this paper is a motivating factor to settle this problem.\\

\noindent{}\textbf{Acknowledgement} \\
The first author sincerely acknowledges the Council of Scientific and Industrial Research (CSIR), India, for providing financial support through the CSIR Fellowship (File No. 09/0028\\(11986)/2021-EMR-I).

\end{document}